\documentclass[11pt, amsfonts]{amsart}

\usepackage{amsmath,amssymb,amsthm}
\usepackage[all]{xy}
\usepackage[colorlinks=true]{hyperref}

\numberwithin{equation}{section}

\SelectTips{eu}{12}

\newtheorem{theorem}{Theorem}[section]
\newtheorem{proposition}[theorem]{Proposition}
\newtheorem{lemma}[theorem]{Lemma}
\newtheorem{corollary}[theorem]{Corollary}

\theoremstyle{definition}
\newtheorem{definition}[theorem]{Definition}
\newtheorem{example}[theorem]{Example}

\newtheorem{problem}[theorem]{Problem}

\theoremstyle{remark}
\newtheorem{remark}[theorem]{Remark}

\newcommand{\Z}{\mathbb{Z}}
\newcommand{\Q}{\mathbb{Q}}
\newcommand{\R}{\mathbb{R}}
\newcommand{\C}{\mathbb{C}}
\newcommand{\Ph}{\mathrm{Ph}}
\newcommand{\F}{\mathcal{F}}
\renewcommand{\lim}{\underset{\longleftarrow}{\mathrm{lim}}}
\newcommand{\limone}{\underset{\longleftarrow}{\mathrm{lim}}^1}

\title[Relative phantom maps]{Relative phantom maps}

\author{Kouyemon Iriye}
\address{Department of Mathematical Sciences, Osaka
Prefecture University, Sakai, 599-8531, Japan}
\email{kiriye@mi.s.osakafu-u.ac.jp}

\author{Daisuke Kishimoto}
\address{Department of Mathematics, Kyoto University, Kyoto, 606-8502, Japan}
\email{kishi@math.kyoto-u.ac.jp}

\author{Takahiro Matsushita}
\address{Department of Mathematical Sciences, University of the Ryukyus, Iribaru-cho, Okinawa 903-0213, Japan}
\email{mtst@sci.u-ryukyu.ac.jp}


\begin{document}

\baselineskip.525cm

\maketitle

\begin{abstract}
The de Bruijn-Erd\H{o}s theorem states that the chromatic number of an infinite graph equals the maximum of the chromatic numbers of finite subgraphs. Such a determinativeness by finite subobjects appears in the definition of a phantom map which is classical in algebraic topology. The topological method in combinatorics connects these two, which leads us to define the relative version of a phantom map: a map $f\colon X\to Y$ is called a relative phantom map to a map $\varphi\colon B\to Y$ if the restriction of $f$ to any finite subcomplex of $X$ lifts to $B$ through $\varphi$, up to homotopy. There are two kinds of maps which are obviously relative phantom maps: (1) the composite of a map $X\to B$ with $\varphi$; (2) a usual phantom map $X\to Y$. A relative phantom map of type (1) is called trivial, and a relative phantom map out of a suspension which is a sum of (1) and (2) is called relatively trivial. We study the (relative) triviality of relative phantom maps and in particular, we give rational homology conditions for the (relative) triviality.
\end{abstract}


\section{Introduction}


\subsection{de Bruijn-Erd\H{o}s theorem}

We start with the classical de Bruijn-Erd\H{o}s theorem on graph colorings. A graph $G$ is called $n$-colorable if its vertices are colored by $n$ colors in such a way that adjacent vertices have different colors. Then the chromatic number $\chi(G)$ is defined to be the minimum $n$ such that $G$ is $n$-colorable. De Bruijn and Erd\H{o}s \cite{dBE} proved the following.

\begin{theorem}
\label{dBE}
The chromatic number of an infinite graph equals the supremum of the chromatic numbers of its finite subgraphs.
\end{theorem}

We shall connect this theorem to algebraic topology. Recall that the index of a free $\Z/2$-space $X$, denoted $\mathrm{ind}(X)$, is the minimum $n$ such that there is a $\Z/2$-equivariant map $X\to S^n$, where $\Z/2$ acts on $S^n$ by the antipodal map. To a graph $G$, possibly infinite, one associates a free $\Z/2$-complex $\mathtt{B}(G)$ which is called the box complex of $G$. As in \cite{MZ}, the chromatic number of a graph $G$ and the index of the box complex $\mathtt{B}(G)$ are related by the inequality
$$\chi(G)\ge\mathrm{ind}(\mathtt{B}(G))+2.$$
Then one may ask whether $\mathrm{ind}(\mathtt{B}(G))$ is determined by the indices of the box complexes of finite subgraphs of $G$ like the chromatic number as in Theorem \ref{dBE}. Let us polish this question to pose a general problem in algebraic topology. First, for any free $\Z/2$-space $X$, there is a $\Z/2$-equivariant map $X\to S^n$ if and only if the classifying map $X/(\Z/2)\to\R P^\infty$ factors through $\R P^n$, up to homotopy. Next, if $G$ is a finite graph, then its box complex $\mathtt{B}(G)$ is a finite complex. Then the above question is generalized  and formalized as the following problem which we call the topological de Bruijn-Erd\H{o}s problem.

\begin{problem}
Suppose that a map $f\colon X\to\R P^\infty$ from a CW-complex $X$ factors through $\R P^n$, up to homotopy, whenever it is restricted to any finite subcomplex of $X$. Then does $f$ itself factor through $\R P^n$?
\end{problem}


\subsection{Relative phantom maps}

Recall that a map $f\colon X\to Y$ from a CW-complex $X$ is called a phantom map if the restriction of $f$ to any finite subcomplex of $X$ is null homotopic. Phantom maps are classical in algebraic topology and their theory has been developed to a quite high level as one can see in \cite{M1}. The determinativeness by finite subobjects of the topological de Bruijn-Erd\H{o}s problem and phantom maps are quite similar, and actually, the topological de Bruijn-Erd\H{o}s problem can be rephrased by the following relative version of phantom maps.

\begin{definition}
A map $f\colon X\to Y$ from a CW-complex $X$ is called a relative phantom map from $X$ to $\varphi\colon B\to Y$ if the restriction of $f$ to any finite subcomplex of $X$ lifts to $B$ through $\varphi$, up to homotopy.
\end{definition}

When $B$ is a point, a relative phantom map to $\varphi\colon B\to Y$ is exactly a usual phantom map, so the name ``relative'' phantom map makes sense. We will call a usual phantom map an absolute phantom map to distinguish it with a relative phantom map.

It is obvious that any map $X\to B$ becomes a relative phantom map from $X$ to $\varphi\colon B\to Y$ after composition with $\varphi$. We call a map $X\to Y$ homotopic to such a composite a trivial relative phantom map from $X$ to $\varphi\colon B\to Y$, which is consistent with the triviality of an absolute phantom map. Then the topological de Bruijn-Erd\H{o}s problem can be rephrased as follows: is there a non-trivial relative phantom map to the inclusion $\R P^n\to\R P^\infty$? Thus we set our aim in this paper to study the ``triviality" of relative phantom maps.


\subsection{Triviality}

Let $\Ph(X,\varphi)$ denote the set of homotopy classes of relative phantom maps from $X$ to $\varphi\colon B\to Y$. We say that $\Ph(X,\varphi)$ is trivial if any relative phantom map from $X$ to $\varphi\colon B\to Y$ is trivial. Note that the triviality of $\Ph (X,\varphi)$ does not imply $\Ph(X, \varphi) = *$. For example, if $\varphi = {\rm id}_B$, then $\Ph(X,\varphi)$ is trivial but $\Ph(X,\varphi) = [X, B]$.

We will first consider a condition equivalent to the (non-)triviality of $\Ph(X,\varphi)$ when $\varphi$ extends to a homotopy fibration $B\xrightarrow{\varphi}Y\to Z$. By using this, we will show the following example which guarantees that there is certainly a non-trivial relative phantom map.

\begin{example}
[Example \ref{example_non-trivial}]
\label{non-trivial_suspension}
Let $u\colon BS^3\to K(\Z,4)$ be a generator of $H^4(BS^3;\Z)\cong\Z$, and extend it to a homotopy fibration sequence 
$$B\xrightarrow{\varphi}Y\to BS^3\xrightarrow{u}K(\Z,4).$$
Then $\Ph(\Sigma\C P^\infty,\varphi)$ is not trivial.
\end{example}

It is well known that any absolute phantom map into a torsion space, that is, a space with $\pi_n$ finite for any $n$, is trivial. Next we will generalize this fact to relative phantom maps. As well as absolute phantom maps \cite{M1}, we consider the class $\F$ of connected CW-complexes having finitely generated $\pi_n$ for $n\ge 2$.

\begin{theorem}
[Proposition \ref{rational_equiv}]
\label{theorem 1}
Suppose that $B,Y\in\F$ and $\varphi\colon B\to Y$ is an isomorphism in $\pi_n\otimes\Q$ for $n\ge 2$. Then $\Ph(\Sigma X,\varphi)$ is trivial.
\end{theorem}


\subsection{Relative triviality}

Any absolute phantom map is obviously a relative phantom map. So relative phantom maps which are trivial relative phantom maps or absolute phantom maps need no special handling. When the source space is a suspension, any sum of such maps can be understood using the existing theory. So we are led to the following definition: a relative phantom map from a suspension $\Sigma X$ to $\varphi\colon B\to Y$ is called {\it relatively trivial} if it is a finite sum of trivial relative phantom maps and absolute phantom maps, and $\Ph(\Sigma X,\varphi)$ is said to be relatively trivial if every relative phantom map from $\Sigma X$ to $\varphi$ is relatively trivial.


\begin{example}
\label{relatively non-trivial}
Let $\varphi\colon B\to Y$ be as in Example \ref{non-trivial_suspension}. Since $Y$ is a torsion space, every phantom map into $Y$ is trivial. Then Example \ref{non-trivial_suspension} shows that there is certainly a relatively non-trivial relative phantom map.
\end{example}

The main theorem of this paper is a further generalization of Theorem \ref{theorem 1} to the relative triviality of relative phantom maps. For a map $\varphi\colon B\to Y$, we put
$$q(\varphi)=\{n\ge 2\,\vert\,\varphi_*\otimes\colon\pi_n(B)\otimes\Q\to\pi_n(Y)\otimes\Q\text{ is not injective}\}.$$
Now we state our main theorem.

\begin{theorem}
[Theorem \ref{criterion_4}]
\label{theorem 2}
Suppose that $B,Y\in\F$ and $H_{n-1}(X;\Q)=0$ for $n\in q(\varphi)$. Then $\Ph(\Sigma X,\varphi)$ is relatively trivial.
\end{theorem}


\subsection{Back to triviality}

When the source space is not a suspension, we cannot consider the relative non-triviality of relative phantom maps. However, when $X$ is not a suspension and any absolute phantom map $X\to Y$ is trivial, the (non-)triviality of relative phantom maps out of $X$ is still our object to study. Let $s_n\colon B\to B_n$ be the $n$-th Postnikov section of a space $B$. Then it is well known that any absolute phantom map into $B_n$ is trivial, so we will study the following problem.

\begin{problem}
\label{generalized-top-dBE}
Find whether or not there is a non-trivial relative phantom map to the Postnikov section $s_n \colon B\to B_n$.
\end{problem}

Of course, the methods developed for relative phantom maps out of a suspension do not apply to Problem \ref{generalized-top-dBE} if the source space is not a suspension. However by a sophisticated consideration on $\limone$, one can prove the following. Put
$$q(B)=\{n\,\vert\,\pi_n(B)\otimes\Q\ne 0\}.$$

\begin{theorem}
\label{theorem 3}
Suppose that $B\in\F$ is nilpotent or has torsion annihilators (see Definition \ref{torsion_annihilator}). If $H_k(X;\Q)=0$ for $k\in q(B)$, then $\Ph(X,s_n)$ is trivial. 
\end{theorem}

We finally return to the topological de Bruijn-Erd\H{o}s problem. Since the inclusion $\R P^n\to\R P^\infty$ is the first Postnikov section of $\R P^n$, Problem \ref{generalized-top-dBE} is actually a generalization of the topological de Bruijn-Erd\H{o}s problem. Since $\R P^n$ is nilpotent for $n$ odd, one gets an answer to the topological de Bruijn-Erd\H{o}s problem by Theorem \ref{theorem 3}.

\begin{corollary}
[Corollary \ref{dBE_answer}]
\label{corollary 1}
If $n$ is odd and $H_n(X ; \Q) = 0$, then $\Ph(X, i_n)$ is trivial.
\end{corollary}

We will construct a space $X(n)$ such that $H_n(X(n);\Q)\ne 0$ and there is a non-trivial phantom map from $X(n)$ to $i_n$. Then Corollary \ref{corollary 1} will turn out to be the optimal answer to the topological de Bruijn-Erd\H{o}s problem in terms of the rational homology of the source space.

\emph{Acknowledgement:} The authors were supported respectively by JSPS KAKENHI (No. 26400094), JSPS KAKENHI (No. 25400087), and JSPS KAKENHI (No. 28-6304). The authors are grateful to the anonymous referee for useful comments.


\section{Relative phantom maps and inverse limits}


\subsection{$\lim$ and $\limone$ of groups}

In this subsection, we recall the definition of $\lim$ and $\limone$ of the inverse system of groups, not necessarily abelian. Let
$$G_0\xleftarrow{f_0}G_1\xleftarrow{f_1}\cdots\xleftarrow{f_{n-1}}G_n\xleftarrow{f_{n}}\cdots$$
be an inverse system of groups, and define the left action of $\prod_{n=0}^\infty G_n$ on itself by
$$(g_0,\ldots,g_n,\ldots)\cdot(x_0,\ldots,x_n,\ldots)=(g_0x_0f_0(g_1)^{-1},\ldots,g_nx_nf_n(g_{n+1})^{-1},\ldots).$$
Then $\lim\,G_n$ and $\limone G_n$ are defined by the isotropy subgroup of $\prod_{n=0}^\infty G_n$ at $(1,1,\ldots)\in \prod_{n=0}^\infty G_n$ and the orbit space of this action, respectively. By definition, $\lim\,G_n$ is a group but $\limone G_n$ is just a pointed set in general whose basepoint is the orbit containing $(1,1, \cdots)$. However, if every $G_n$ is abelian, then $\limone G_n$ has a natural abelian group structure.

Next we recall the 6-term exact sequence (Lemma \ref{6-term}) involving $\lim$ and $\limone$ which will be useful. For a basepoint preserving map $h\colon S\to T$ between pointed sets, we write $\mathrm{Im}\,h$ and $\mathrm{Ker}\,h$ to mean $h(S)$ and $h^{-1}(*)$, respectively. Recall that a sequence of pointed sets $A\xrightarrow{f}B\xrightarrow{g}C$ is exact if $\mathrm{Im}\,f=\mathrm{Ker}\,g$. If $A,B,C$ are groups and $f,g$ are group homomorphisms, the exactness coincides with that of groups.

\begin{lemma}
\label{6-term}
Let $1\to\{G_n\}\to\{H_n\}\to\{K_n\}\to 1$ be an exact sequence of inverse systems of groups. Then there is a natural exact sequence of pointed sets:
$$1\to\lim\,G_n\to\lim\,H_n\to\lim\,K_n\to\limone G_n\to\limone H_n\to\limone K_n\to*$$
\end{lemma}


\subsection{Absolute phantom maps}

Recall that a map $f\colon X \rightarrow Y$ from a CW-complex is a phantom map if the restriction of $f$ to any finite subcomplex of $X$ is null homotopic. Then $f\colon X\to Y$ is a phantom map if and only if the composite $f\circ g$ is null homotopic for any map $g\colon K\to X$ from a finite complex $K$. This implies that the definition of a phantom map does not depend on a particular cell structure of the source space. Hereafter, we will assume that the source space of a phantom map is a CW-complex of finite type. Then in particular, $f\colon X\to Y$ is a phantom map if and only if the restriction of $f$ to any finite dimensional skeleton of $X$ is null homotopic. As mentioned in Section 1, we will call a usual phantom map an absolute phantom map to distinguish it from relative phantom maps. Let $\Ph(X,Y)$ denote the set of homotopy classes of absolute phantom maps from $X$ to $Y$.

Let $X^n$ denote the $n$-skeleton of a CW-complex $X$. By the Milnor exact sequence (see \cite{BK})
\begin{equation}
\label{milnor}
*\to\limone[\Sigma X^n,Y]\to[X,Y]\xrightarrow{\pi_Y} \lim\,[X^n,Y]\to*
\end{equation}
we have the following description of $\Ph(X,Y)$ by $\limone$.

\begin{proposition}
\label{absolute_Ph}
There is an isomorphism of pointed sets
$$\Ph(X,Y)\cong\limone[\Sigma X^n,Y],$$
which is a group isomorphism whenever $X$ is a suspension.
\end{proposition}

We can dualize this proposition by considering the Postnikov tower of the target space, where the proof is omitted. Let $Y_n$ denote the $n$-th Postnikov section of $Y$, and let $Y_1 \leftarrow Y_2 \leftarrow \cdots \leftarrow Y_n \leftarrow\cdots$ be the Postnikov tower of $Y$.

\begin{proposition}
\label{absolute_Ph_dual}
There is an isomorphism of pointed sets
$$\Ph(X,Y)\cong\limone[X,\Omega Y_n],$$
which is a group isomorphism whenever $X$ is a suspension.
\end{proposition}

We record consequences of the two propositions above on the triviality of $\Ph(X,Y)$ that we are going to use. As in Section 1, let $\F$ denote the class of connected CW-complexes having finitely generated $\pi_n$ for $n\ge 2$.

\begin{corollary}
\label{Ph=0}
\begin{enumerate}
\item If $Y$ is a finite Postnikov section, then $\Ph(X,Y)=*$.
\item If $Y\in\F$ satisfies that $\pi_*(Y)\otimes\Q=0$ for $*\ge 2$, then $\Ph(X,Y)=*$.
\end{enumerate}
\end{corollary}

\begin{proof}
(1) is immediate from Proposition \ref{absolute_Ph_dual}. (2) For any finite connected complex $A$, the homotopy set $[\Sigma A,Y]$ is a finite set by the assumption on $Y$, and the inverse system of finite groups satisfies the Mittag-Leffler condition (see \cite{M1}). Then $\Ph(X,Y)\cong \limone[\Sigma X^n,Y]=*$.
\end{proof}


\section{Relative phantom maps}

In Section 1, we have defined that a map $X\to Y$ from a CW-complex $X$ is a relative phantom to $\varphi\colon B\to Y$ if the restriction of $f$ to any finite subcomplex of $X$ lifts to $B$ through $\varphi$, up to homotopy. As well as absolute phantom maps, one can see that the definition of a relative phantom map does not depend on a particular cell structure of the source space. Just as for absolute phantom maps, we will always assume that the source space of a relative phantom map is a connected CW-complex of finite type. In particular, $f\colon X\to Y$ is a relative phantom map to $\varphi\colon B\to Y$ if and only its restriction to any finite dimensional skeleton lifts to $B$ through $\varphi$, up to homotopy.

Analogous to the absolute case in Proposition \ref{absolute_Ph_dual}, let us dualize the definition of relative phantom maps. Let $Y_1 \leftarrow Y_2 \leftarrow \cdots \leftarrow Y_n \leftarrow\cdots$ be the Postnikov tower of $Y$ as in the previous section, and let $s_n\colon Y\to Y_n$ be the $n$-th Postnikov section of $Y$. By the naturality of Postnikov towers, a map $\varphi\colon B\to Y$ induces a map $\varphi_n\colon B_n\to Y_n$ between the Postnikov sections satisfying $\varphi_n\circ s_n^B\simeq s_n^Y\circ\varphi$, where $s_n^B$ and $s_n^Y$ are the Postnikov sections of $B$ and $Y$, respectively. 

\begin{proposition}
The following conditions on a map $f \colon X \to Y$ are equivalent:
\begin{enumerate}
\item $f$ is a relative phantom map to $\varphi$;
\item For any $n\ge 0$, $s_n \circ f \colon X \to Y_n$ has a lift with respect to $\varphi_n \colon B_n\to Y_n$, up to homotopy.
\end{enumerate}
\end{proposition}

\begin{proof}
Suppose that $f$ is a relative phantom map to $\varphi$. We want to show that $s_n \circ f \colon X \to Y_n$ has a lift with respect to $\varphi_n$, up to homotopy, for any $n$. Since $f$ is a relative phantom map to $\varphi$, the map $f\vert_{X^{n+1}}\colon X^{n+1}\to Y$ has a lift $\tilde{f}\colon X^{n+1}\to B$ through $\varphi$, up to homotopy. Since the inclusion $X^{n+1}\to X$ induces an isomorphism $[X,B_n]\xrightarrow{\cong}[X^{n+1},B_n]$ of pointed sets, there is a map $\bar{f}\colon X\to B_n$ satisfying $\bar{f}\vert_{X^{n+1}}\simeq s_n\circ\tilde{f}$. Now we have
$$\varphi_n\circ\bar{f}\vert_{X^{n+1}}\simeq\varphi_n\circ s_n\circ\tilde{f}\simeq s_n\circ\varphi\circ\tilde{f}\simeq s_n\circ f\vert_{X^{n+1}}.$$
Since the inclusion $X^{n+1}\to X$ induces an isomorphism $[X,Y_n]\xrightarrow{\cong}[X^{n+1},Y_n]$ as pointed sets, we obtain that $\varphi_n\circ\bar{f}\simeq s_n\circ f$. Thus $\bar{f}$ is a desired lift.

Suppose next that for any $n$, $s_{n+1} \circ f \colon X \to Y_{n+1}$ has a lift $g\colon X\to B_{n+1}$ with respect to $\varphi_{n+1}$, up to homotopy. We want to show that $f|_{X^n} \colon X^n \to Y$ has a lift with respect to $\varphi$, up to homotopy. Since there is an isomorphism $(s_{n+1})_*\colon[X^n,B]\xrightarrow{\cong}[X^n,B_{n+1}]$ of pointed sets, we have a map $\bar{g}\colon X^n\to B$ satisfying $s_{n+1}\circ\bar{g}\simeq g\vert_{X^n}$. Then we get
$$s_{n+1}\circ\varphi\circ \bar{g}\simeq\varphi_{n+1}\circ s_{n+1}\circ\bar{g}\simeq\varphi_{n+1}\circ g\vert_{X^n}\simeq s_{n+1}\circ f\vert_{X^n}.$$
Since the map $(s_{n+1})_*\colon[X^n,Y]\to[X^n,Y_{n+1}]$ is also isomorphic, we get $\varphi\circ\bar{g}\simeq f\vert_{X^n}$ as required.
\end{proof}

Next we give a description of $\Ph(X,\varphi)$ by using $\Ph(X,Y)$ which will be useful to deal with $\Ph(X,\varphi)$ algebraically.

\begin{proposition}
\label{Ker_I}
There is an exact sequence of pointed sets
$$1 \to \Ph(X, Y) \to \Ph(X,\varphi) \xrightarrow{\pi_Y} \lim \, \varphi_* [X^n,B] \to 1$$
which is an exact sequence of groups whenever $X$ is a suspension.
\end{proposition}

\begin{proof}
Note that an element $f$ of $[X,Y]$ is a relative phantom map to $\varphi$ if and only if $\pi_Y(f) \in \lim [X^n, Y]$ is contained in $\lim\, \varphi_* [X^n, B]$. This means that the diagram
$$\xymatrix{
\Ph(X, \varphi) \ar[d] \ar[r]^-{\pi_Y} & \lim \, \varphi_* [X^n,B] \ar[d]\\
[X,Y] \ar[r]^-{\pi_Y} \ar[r] & \lim [X^n,Y]
}$$
is a pullback. By the Milnor exact sequence \eqref{milnor}, the lower $\pi_Y$ is surjective, implying that the upper $\pi_Y$ is surjective too. By \eqref{milnor}, we also have that the kernel of the lower $\pi_Y$ is $\limone[\Sigma X^n,Y]$. Thus the kernel of the upper $\pi_Y$ is isomorphic to $\limone[\Sigma X^n,Y]$ which is isomorphic with $\Ph(X,Y)$ by Proposition \ref{absolute_Ph}, completing the proof.
\end{proof}


\section{Triviality of relative phantom maps out of a suspension}

A relative phantom map $f\colon X\to Y$ to $\varphi\colon B\to Y$ is called trivial if the entire map $f$ has a lift with respect to $\varphi$, up to homotopy, and $\Ph(X, \varphi)$ is called trivial if every element of $\Ph(X, \varphi)$ is trivial. We consider the triviality of relative phantom maps to $\varphi\colon B\to Y$ when $\varphi$ is a fiber inclusion, that is, there is a homotopy fibration $B\xrightarrow{\varphi}Y\to W$. This case descends to relative phantom maps out of a suspension as follows. Given a map $\varphi \colon B\to Y$, there is a homotopy fibration $\Omega B \xrightarrow{\Omega\varphi}\Omega Y \to F$, where $F$ is the homotopy fiber of $\varphi$. Then $\Omega \varphi$ is a fiber inclusion and by the adjointness, we have
\begin{equation}
\label{Ph_adjoint}
\Ph(\Sigma X,\varphi)\cong\Ph(X,\Omega\phi).
\end{equation}

The following proposition enables us to detect the (non-)triviality of relative phantom maps by that of related absolute phantom maps.

\begin{proposition}
\label{triviality}
Let $B \xrightarrow{\varphi} Y \xrightarrow{p} Z$ be a homotopy fibration. Then a map $f \colon X \to Y$ is a relative phantom map to $\varphi$ if and only if the composite $p \circ f \colon X \to Z$ is an absolute phantom map. Moreover, $f$ is a trivial relative phantom map if and only if $p \circ f$ is null homotopic.
\end{proposition}

\begin{proof}
For every $n$, $f|_{X^n}$ has a lift with respect to $\varphi$, up to homotopy, if and only if $p \circ f|_{X^n}$ is null homotopic. This implies that $f$ is a relative phantom map to $\varphi$ if and only if $p \circ f$ is an absolute phantom map. Similarly, $p \circ f$ is null homotopic if and only if $f$ has a lift with respect to $\varphi$, up to homotopy. Thus the proof is done.
\end{proof}

We show two applications of Proposition \ref{triviality}. The first one is as follows. We denote the adjoint of a map $f\colon\Sigma X\to Y$ by $\mathrm{ad}(f)\colon X\to\Omega Y$. 

\begin{corollary}
\label{suspension triviality 1}
Let $\Omega Y \xrightarrow{\delta} F \to B \xrightarrow{\varphi} Y$ be a homotopy fibration sequence. A map $f\colon\Sigma X\to Y$ is a relative phantom map to $\varphi$ if and only if $\delta \circ \mathrm{ad}(f)$ is an absolute phantom map. Moreover, $f$ is trivial if and only if $\delta \circ \mathrm{ad}(f)$ is null homotopic.
\end{corollary}

\begin{proof}
Note that $f\colon \Sigma X \rightarrow Y$ is a (trivial) relative phantom map to $\varphi$ if and only if $\mathrm{ad}(f) \colon X \rightarrow \Omega Y$ is a (trivial) relative phantom map to $\Omega \varphi$. Then by applying Proposition \ref{triviality} to the fibration sequence $\Omega B\xrightarrow{\Omega\varphi}\Omega Y\xrightarrow{\delta}F$, the proof is done.
\end{proof}

\begin{corollary}
\label{suspension triviality 2}
Suppose that we have a homotopy fibration $F\to B\xrightarrow{\varphi}Y$ such that the connecting map $\delta\colon\Omega Y\to F$ is null homotopic. Then $\Ph(\Sigma X,\varphi)$ is trivial.
\end{corollary}

\begin{proof}
Since $\delta$ is null homotopic, so is $\delta\circ\mathrm{ad}(f)$ for any $f\in\Ph(\Sigma X,\varphi)$. Then by Corollary \ref{suspension triviality 1}, $f$ is trivial, completing the proof.
\end{proof}

\begin{example}
\label{Ganea_fibration}
Let $F_n(Y)\to G_n(Y)\xrightarrow{p_n}Y$ be the $n$-th Ganea fibration. Since $\Omega G_n(Y) \to \Omega Y$ has a section (see Chapter 1 of \cite{CLOT}), the connecting map $\delta \colon \Omega Y \to F_n(Y)$ is null homotopic. Thus Corollary \ref{suspension triviality 2} implies that $\Ph(\Sigma X, p_n)$ is trivial.



Although we have seen that $\Ph(\Sigma X,p_n)$ is trivial, we will see in Proposition \ref{dBE_non-trivial} below that there is a non-suspension space $X(n)$ such that $\Ph(X(n),p_n)$ is not trivial for $Y=\R P^\infty$ with $n>2$.
\end{example}

The next lemma is a variant of Corollary \ref{suspension triviality 1} and will be used to prove Proposition \ref{rational_equiv} below which is a generalization of Corollary \ref{Ph=0} to the relative case.

\begin{lemma}
\label{suspension triviality 3}
Let $F$ be the homotopy fiber of a map $\varphi\colon B\to Y$. Then $\Ph(\Sigma X,\varphi)$ is trivial whenever $\Ph(X, F) = *$.
\end{lemma}

\begin{proof}
Let $\delta\colon\Omega Y\to F$ be the connecting map of a homotopy fibration $F\to B\xrightarrow{\varphi}Y$. Since $\Ph(X, F) = *$, $\delta\circ\mathrm{ad}(f)$ is a trivial absolute phantom map for any $f\in\Ph(\Sigma X,\varphi)$. Then $f$ is trivial by Corollary \ref{suspension triviality 1}, completing the proof.
\end{proof}

As in Section 1, we will write $\F$ to denote the class of connected CW complexes each of which has finitely generated $\pi_n$ for $n\ge 2$.

\begin{proposition}
\label{rational_equiv}
Let $B,Y\in\F$. Suppose that $\varphi \colon B \to Y$ is an isomorphism in $\pi_n\otimes\Q$ for $n\ge 2$. Then $\Ph(\Sigma X,\varphi)$ is trivial.
\end{proposition}

\begin{proof}
By the assumption, the homotopy fiber $F$ of $\varphi$ satisfies the condition of Corollary \ref{Ph=0}, implying $\Ph(X,F)=*$. Then we get the desired result by Lemma \ref{suspension triviality 3}.
\end{proof}

Next we show the second application of Proposition \ref{triviality}.

\begin{proposition}
\label{non-trivial}
Suppose that there is a homotopy fibration sequence $B \xrightarrow{\varphi} Y \xrightarrow{\alpha} V \xrightarrow{\beta} W$ such that either $\beta$ is null homotopic or $\Ph(X,W)=*$. Then $\Ph(X, \varphi)$ is trivial if and only if $\Ph(X, V)=*$. 
\end{proposition}

\begin{proof}
It clearly follows from Proposition \ref{triviality} that $\Ph(X,V) = *$ implies that $\Ph(X, \varphi)$ is trivial. On the other hand, let $f \colon X \to V$ be an absolute phantom map. Then $\beta\circ f\colon X\to W$ is an absolute phantom map, so by the assumption, $\beta \circ f$ is null homotopic. Thus $f$ has a lift $\tilde{f}$ with respect to $\alpha$, up to homotopy. By Proposition \ref{triviality}, $\tilde{f}$ is a relative phantom map from $X$ to $\varphi$ which is trivial if and only if $f\colon X\to V$ is null homotopic. Therefore the proof is completed.
\end{proof}

\begin{corollary}
\label{non-trivial 2}
Let $F \xrightarrow{j} B\xrightarrow{\varphi}Y$ be a homotopy fibration such that either $j$ is null homotopic or $\Ph(X,B)=*$. Then $\Ph(\Sigma X, \varphi)$ is trivial if and only if $\Ph(X, F)$ is trivial.
\end{corollary}

\begin{proof}
Apply Proposition \ref{non-trivial} to the homotopy fibration sequence $\Omega B \xrightarrow{\Omega \varphi} \Omega Y \xrightarrow{s} F \xrightarrow{j} B$ together with the adjoint congruence \eqref{Ph_adjoint}.
\end{proof}

\begin{example}
\label{example_non-trivial}
We give an example of a space $X$ and a map $\varphi$ such that $\Ph(\Sigma X,\varphi)$ is non-trivial although $\Ph(\Sigma X, Y)$ is trivial. Let $u\colon BS^3\to K(\Z,4)$ be a generator of $H^4(BS^3;\Z)\cong\Z$, and extend it to a homotopy fibration sequence 
$$S^3\xrightarrow{\Omega u}K(\Z,3)=B\xrightarrow{\varphi}Y\to BS^3\xrightarrow{u}K(\Z,4).$$
By Corollary \ref{Ph=0}, we have $\Ph(X,B)=*$ for any space $X$. So we can apply Corollary \ref{non-trivial 2} to the homotopy fibration sequence $S^3\xrightarrow{\Omega u}K(\Z,3)=B\xrightarrow{\varphi}Y$. By \cite{G}, we have $\Ph(\C P^\infty,S^3)\ne *$, and thus we obtain that $\Ph(\Sigma\C P^\infty,\varphi)$ is not trivial. On the other hand, it follows from Corollary \ref{Ph=0} that $\Ph(\Sigma \C P^\infty, Y)$ is trivial.
\end{example}


\section{Relative triviality of relative phantom maps out of a suspension}

Any absolute phantom map is a relative phantom map and it is not an object that we would like to study in this paper. So as in Section 1, a relative phantom map out of a suspension is called relatively trivial if it is a finite sum of trivial relative phantom maps and absolute phantom maps, and we will investigate conditions for the existence of a relatively non-trivial phantom map. By Example \ref{relatively non-trivial}, there is certainly a relatively non-trivial phantom map. We say that $\Ph(\Sigma X,\varphi)$ is relatively trivial if it consists only of relatively trivial relative phantom maps. We first observe basic properties of relatively trivial relative phantom maps. Note that the set $\Ph(\Sigma X,\varphi)$ is a group.

\begin{proposition}
\begin{enumerate}
\item Any relatively trivial relative phantom map is homotopic to the sum $f+g$, where $f$ is a trivial phantom map and $g$ is an absolute phantom map.
\item The set of relatively trivial relative phantom maps from $\Sigma X$ to $\varphi\colon B\to Y$ is a subgroup of $\Ph(\Sigma X,\varphi)$.
\end{enumerate}
\end{proposition}

\begin{proof}
The map $\pi_Y\colon\Ph(\Sigma X,\varphi)\to\lim\, \varphi_* [\Sigma X^n, B]$ in Proposition \ref{Ker_I} is a group homomorphism whose kernel is $\Ph(\Sigma X,Y)$. In particular, $\Ph(\Sigma X,Y)$ is a normal subgroup of $\Ph(\Sigma X,\varphi)$, implying (1). We also get that the set of relatively trivial relative phantom maps from $X$ to $\varphi$ is the subgroup $\varphi_*[\Sigma X,B]+\Ph(\Sigma X,Y)$ of $\Ph(\Sigma X,\varphi)$. Thus the proof is done.
\end{proof}

We investigate conditions which guarantee that $\Ph(\Sigma X,\varphi)$ is relatively trivial. 

\begin{lemma}
$\Ph(\Sigma X, \varphi)$ is relatively trivial if and only if the composite 
$$[\Sigma X, B] \xrightarrow{\varphi_*} \Ph(\Sigma X,\varphi) \xrightarrow{\pi_Y} \lim\, \varphi_* [\Sigma X^n, B]$$
is surjective, where the map $\pi_Y$ is as Proposition \ref{Ker_I}.
\end{lemma}

\begin{proof}
Suppose first that $\Ph(\Sigma X,\varphi)$ is relatively trivial. There is a commutative diagram of groups
\begin{equation}
\label{rel_trivial_diagram}
\xymatrix{[\Sigma X,B]\ar[d]^{\varphi_*}\ar[r]^-{\pi_B}&\lim\,[\Sigma X^n,B]\ar[d]^{\varphi_*}\\
\Ph(\Sigma X,\varphi)\ar[r]^-{\pi_Y} & \lim\, \varphi_* [\Sigma X^n, B]}
\end{equation}
where $\pi_B$ and $\pi_Y$ denotes the natural projections as in \eqref{milnor} and Proposition \ref{Ker_I}. Then by Proposition \ref{Ker_I}, the bottom arrow $\pi_Y$ of \eqref{rel_trivial_diagram} is surjective, so for any $f\in\lim \, \varphi_* [\Sigma X^n, B]$, there is $\tilde{f} \in \Ph(\Sigma X, \varphi)$ satisfying $\pi_Y (\tilde{f}) = f$. By the assumption, $\tilde{f}$ is relatively trivial, so there are $g\in[\Sigma X,B]$ and $h\in\Ph(\Sigma X,Y)$ such that $\tilde{f}=\varphi_*(g)+h$. Now we have
$$f=\pi_Y(\tilde{f})=\pi_Y(\varphi_*(g)+h)=\pi_Y\circ\varphi_*(g)+\pi_Y(h)$$
where $\pi_Y$ is a group homomorphism. By definition, we have $\pi_Y(h)=0$, and then we have proved that $\pi_Y \circ \varphi_*$ is surjective.

Next suppose that $\pi_Y \circ \varphi_*$ is surjective, and take any $f \in \Ph (\Sigma X, \varphi)$. Then there is $g\in[\Sigma X,B]$ such that $\pi_Y\circ\varphi_*(g)=\pi_Y(f)$, implying $f-\varphi_*(g)\in\mathrm{Ker}\,\pi_Y$. Since $\mathrm{Ker}\,\pi_Y=\Ph(\Sigma X,Y)$ by Proposition \ref{Ker_I}, there is $h\in\Ph(\Sigma X,Y)$ satisfying $f-\varphi_*(g)=h$, or equivalently, $f=h +\varphi_*(g)$. Thus $\Ph(\Sigma X,\varphi)$ is relatively trivial. Therefore the proof is completed.
\end{proof}

Let $K_n$ be the kernel of the group homomorphism $\varphi_* \colon [\Sigma X^n, B] \to [\Sigma X^n, Y]$. Then we have the following exact sequence of inverse systems of groups:
$$1 \longrightarrow \{ K_n\} \longrightarrow \{ [\Sigma X^n, B]\} \longrightarrow \{ \varphi_* [\Sigma X^n, B]\} \longrightarrow 1$$

\begin{proposition}
\label{criterion_1}
$\Ph(\Sigma X,\varphi)$ is relatively trivial if and only if the kernel of the map 
$$\limone K_n\to\limone[\Sigma X^n,B]$$
is trivial.
\end{proposition}

\begin{proof}
Consider the commutative diagram \eqref{rel_trivial_diagram}. Since the top map $\pi_B$ is surjective by the Milnor exact sequence \eqref{milnor}, the map $\varphi_*\circ\pi_B\colon[\Sigma X,B]\to\lim\, \varphi_* [\Sigma X^n, B]$ is surjective if and only if $\varphi_*\colon\lim\,[\Sigma X^n,B]\to\lim\, \varphi_* [\Sigma X^n, B]$ is surjective. Applying Lemma \ref{6-term} to the short exact sequence
$$1\to\{K_n\}\to\{[\Sigma X^n,B]\}\xrightarrow{\varphi_*}\{\varphi_* [\Sigma X^n,B]\}\to 1$$
of inverse systems of groups, we get an exact sequence
$$\lim\,[\Sigma X^n,B]\xrightarrow{\varphi_*}\lim\, \varphi_*[\Sigma X^n, B] \to\limone K_n\to\limone[\Sigma X^n,B]$$
of pointed sets. Thus the map $\varphi_*\colon\lim\,[\Sigma X^n,B]\to\lim\, \varphi_* [\Sigma X^n, B]$ is surjective if and only if the kernel of the map $\limone K_n\to\limone[\Sigma X^n,B]$ is trivial. This completes the proof.
\end{proof}

The assumption of the following corollary trivially implies that of Proposition \ref{criterion_1}.

\begin{corollary}
\label{criterion_2}
$\Ph(\Sigma X,\varphi)$ is relatively trivial whenever $\limone K_n=*$.
\end{corollary}

We then consider practical conditions which guarantee $\limone K_n=*$. We first translate the condition $\limone K_n=*$ to that of absolute phantom maps.

\begin{lemma}
\label{criterion_3}
Let $F\xrightarrow{j}B\xrightarrow{\varphi}Y$ be a homotopy fibration with the connecting map $\delta\colon\Omega Y\to F$. For any space $X$, $\limone K_n=*$ if and only if the map $\delta_* \colon \Ph(X,\Omega Y)\to\Ph(X,F)$ is surjective.
\end{lemma}

\begin{proof}
Put $L_n=\mathrm{Ker}\{j_* \colon[\Sigma X^n, F]\to[\Sigma X^n,B]\}$. By the exactness of the sequence
$$[\Sigma X^n, F] \xrightarrow{j_*} [\Sigma X^n,B] \xrightarrow{\varphi_*} [\Sigma X^n,Y],$$
we have an exact sequence of inverse systems of groups
$$1 \to \{ L_n\} \to \{ [\Sigma X^n,F]\} \to \{ K_n\} \to 1.$$
Then by Lemma \ref{6-term}, we get an exact sequence of pointed sets
$$\limone L_n\to\limone[\Sigma X^n,F]\to\limone K_n\to *.$$
Thus $\limone K_n = *$ if and only if the map $\limone L_n \rightarrow \limone [\Sigma X^n, F]$ is surjective.

Next we put $M_n = \mathrm{Ker} \{ \delta_* \colon [\Sigma X^n, \Omega Y] \to [\Sigma X^n, F]\}$. Similarly to the above, from the exact sequence of groups
$$[\Sigma X^n , \Omega Y] \xrightarrow{\delta_*} [\Sigma X^n, F] \xrightarrow{j_*} [\Sigma X^n, Y],$$
we get an exact sequence of inverse systems of groups
$$1 \to \{ M_n \} \to \{ [\Sigma X^n, \Omega Y] \} \to \{ L_n \} \to 1.$$
Thus by Lemma \ref{6-term}, we have that  $\limone [\Sigma X^n, \Omega Y] \to \limone L_n$ is surjective. Then $\limone K_n=*$ if and only if the composite $\limone [\Sigma X^n, \Omega Y] \to \limone L_n\to\limone [\Sigma X^n, F]$ is surjective. By Proposition \ref{absolute_Ph}, this composite is identified with $\delta_*\colon\Ph(X,\Omega Y)\to\Ph(X,F)$. Thus the proof is completed.
\end{proof}

As we have given a rational homotopy condition for the triviality of $\Ph(\Sigma X,\varphi)$ in Proposition \ref{rational_equiv}, we expect to find a rational homotopy condition for the relative triviality of $\Ph(\Sigma X,\varphi)$. McGibbon and Roitberg \cite{MR} gave a necessary and sufficient rational homotopy condition which guarantees that every phantom map $X\to Y$ is null homotopic, and we are motivated by their result to consider a rational homotopy condition for the relative triviality of $\Ph(\Sigma X,\varphi)$. We first recall the result of Roitberg and Touhey \cite{RT}.

\begin{theorem}
[Roitberg and Touhey \cite{RT}]
\label{RT}
For $Y\in\F$, there is an isomorphism of pointed sets
\begin{equation}
\label{Ph_cohomology}
\Ph(X,Y)\cong\prod_{n\ge 1}H^n(X;\pi_{n+1}(Y)\otimes\widehat{\Z}/\Z)/[X,\Omega\widehat{Y}]
\end{equation}
which is natural with respect to $X$ and $Y$, where $\widehat{\Z}$ is the profinite completion of the integer ring $\Z$ and $\widehat{Y}$ is the profinite completion of a space $Y$ in the sense of Sullivan.
\end{theorem}

\begin{remark}
Although more conditions on $Y$ are assumed in \cite{RT}, we may replace $Y$ with its universal cover by Proposition \ref{absolute_Ph} so that the conditions reduce to that $Y\in\F$.
\end{remark}

Next we apply Theorem \ref{RT} to the induced map between absolute phantom maps. For a map $g\colon V\to W$, we put
$$\hat{q}(g)=\{n\ge 2\,\vert\,g_* \colon \pi_n(V)\otimes\Q \to \pi_n(W)\otimes\Q\text{ is not surjective}\}.$$

\begin{lemma}
\label{Ph_surj}
Given a map $g\colon V\to W$ for $V,W\in\F$, suppose that $H_{n-1}(X ; \Q) = 0$ for $n\in\hat{q}(g)$. Then $g_*\colon \Ph(X,V) \to \Ph(X,W)$ is surjective.
\end{lemma}

\begin{proof}
Since the isomorphism of Theorem \ref{RT} is natural with respect to $Y$, the lemma immediately follows from the fact that $\widehat{\Z}/\Z$ is a $\Q$-vector space.
\end{proof}

Put 
$$q(\varphi)=\{n\ge 2\,\vert\,\varphi_*\colon\pi_{n}(B)\otimes\Q\to\pi_{n}(Y)\otimes\Q\text{ is not injective}\}.$$
Now we give a rational homotopy condition for the relative triviality of $\Ph(\Sigma X,\varphi)$.

\begin{theorem}
\label{criterion_4}
Let $B,Y\in\F$. If $H_{n-1}(X;\Q)=0$ for $n\in q(\varphi)$, then $\Ph(\Sigma X, \varphi)$ is relatively trivial. 
\end{theorem}

\begin{proof}
Let $F$ be the homotopy fiber of $\varphi\colon B\to Y$ and $\delta\colon\Omega Y\to F$ be the corresponding connecting map. By the homotopy exact sequence, $\pi_n(\Omega Y) \otimes \Q \to \pi_n(F) \otimes \Q$ is surjective if and only if $\varphi_* \colon \pi_n(B) \otimes \Q \to \pi_n(Y) \otimes \Q$ is injective for $n\ge 2$. Then we have $q(\varphi)=\hat{q}(\delta)$. Thus the proof is completed by Corollary \ref{criterion_2} and Lemmas \ref{criterion_3} and \ref{Ph_surj}. 
\end{proof}

We give three corollaries of this theorem.

\begin{corollary}
Let $B,Y\in\F$. If $\varphi_* \colon \pi_n(B) \otimes \Q \to \pi_n(Y) \otimes \Q$ is injective for $n \ge 2$, then $\Ph(\Sigma X, \varphi)$ is relatively trivial.
\end{corollary}

For a space $A$, we put 
$$q(A)=\{n\ge 2\,\vert\,\pi_n(A)\otimes\Q\ne 0\}.$$

\begin{corollary}
\label{criterion_5}
Let $B,Y\in\F$. If $H_{n-1}(X;\Q)=0$ for $n\in q(F)$, then $\Ph(\Sigma X, \varphi)$ is relatively trivial, where $F$ is the homotopy fiber of $\varphi\colon B\to Y$.
\end{corollary}

\begin{proof}
By the homotopy exact sequence of the homotopy fibration $F\to Y\xrightarrow{\varphi}B$, we see that $q(\varphi)\subset q(F)$. Thus the proof is done by Theorem \ref{criterion_4}.
\end{proof}

\begin{corollary}
\label{criterion_6}
Let $B,Y\in\F$ and $F\xrightarrow{j}B\xrightarrow{\varphi}Y$ be a homotopy fibration such that $j$ is null homotopic. Then $\Ph(\Sigma X,\varphi)$ is relatively trivial.
\end{corollary}

We close this section with the following example.

\begin{example}
By definition, if $\Ph(\Sigma X,\varphi)$ is trivial, then it is relatively trivial. Here we give a space $X$ and a map $\varphi$ such that the converse of this implication does not hold, that is, $\Ph(\Sigma X,\varphi)$ is relatively trivial and is not trivial.

Let $S^3\to S^{4n+3}\xrightarrow{p_n}\mathbb{H}P^n$ be the Hopf fibration. Since the fiber inclusion $S^3\to S^{4n+3}$ is null homotopic, $\Ph(\Sigma X,p_n)$ is relatively trivial by Corollary \ref{criterion_6}. By Corollary \ref{non-trivial 2}, we also have that $\Ph(\Sigma X,p_n)$ is trivial if and only if $\Ph(X,S^3)=*$. Then since $\Ph(\C P^\infty,S^3)\ne*$ by \cite{G}, we get that $\Ph(\Sigma\C P^\infty,p_n)$ is not trivial. Thus we have obtained that $\Ph(\Sigma\C P^\infty,p_n)$ is relatively trivial and is not trivial.
\end{example}


\section{Triviality of relative phantom maps out of a non-suspension}

In this section, we consider Problem \ref{generalized-top-dBE}. By Corollary \ref{Ph=0}, we have $\Ph(X,B_n)=*$ for all $X$, so the triviality and the relative triviality of phantom maps out of a suspension to $s_n \colon B\to B_n$ are the same. The case of relative phantom maps out of a suspension in Problem \ref{generalized-top-dBE} has been studied in the previous sections. In particular, by Example \ref{Ganea_fibration}, $\Ph(\Sigma X,i_n)$ is trivial for the inclusion $i_n\colon\R P^n\to\R P^\infty$. Thus we consider relative phantom maps out of a non-suspension for Problem \ref{generalized-top-dBE}. When $X$ is not a suspension, the Puppe exact sequence associated with skeleta of $X$ is not an exact sequence of groups, so we cannot use Lemma \ref{6-term} which has been fundamental in many places above. Instead, we will use the following lemma.

\begin{lemma}
[cf. {\cite[Lemma 1.1.5]{RZ}}]
\label{surj}
Let $\{f_n\}\colon\{G_n\}\to\{H_n\}$ be a continuous map between inverse systems of compact Hausdorff topological spaces. Then the map $\lim\,f_n\colon\lim\,G_n\to\lim\,H_n$ is surjective whenever each $f_n\colon G_n\to H_n$ is so.
\end{lemma}

Let $V$ be a finite complex and $W$ be a torsion space, that is, $\widetilde{H}_n(W;\Q)=0$ for any $n$. Then it is well known that the homotopy set $[V,W]$ is finite. We generalize this fact in two cases. The first case is the following.

\begin{lemma}
\label{finite1}
If $B\in\F$ is nilpotent with finite $\pi_1$ and a finite complex $Z$ satisfies $H_k(Z;\Q)=0$ for $k\in q(B)$, then $[Z,B]$ is finite.
\end{lemma}

\begin{proof}
Let $\cdots\xrightarrow{q_{k+1}} B(k+1) \xrightarrow{q_k} B(k) \xrightarrow{q_{k-1}}\cdots\xrightarrow{q_0} B(0) = *$ be a principal replacement of the Postnikov tower of $B$. Since $Z$ is a finite complex, we have $[Z,B]\cong[Z,B(k)]$ for large $k$. Then it suffices to show that $[Z, B(k)]$ is finite for any $k$. We prove this by induction on $k$.

Each arrow $q_k\colon B(k+1) \to B(k)$ is a principal fibration with fiber $K(A_k,m_k)$ such that $A_k$ is an abelian group. Then we have an exact sequence of pointed sets
$$H^{m_k}(Z;A_k) \to [Z,B(k)]\xrightarrow{(q_{k-1})_*}[Z,B(k-1)].$$
Since $q_{k-1}\colon B(k) \to B(k-1)$ is principal, we have $|(q_{k-1})_*^{-1}(a)|\le|H^{m_k}(Z;A_k)|$ for any $a\in[Z,B(k-1)]$. Moreover, by the assumption on $X$, $\widetilde{H}^{m_k}(Z;A_k)$ is finite for any $k$. Then the proof is done by induction on $k$ starting with $[Z,B(0)]=*$ for $B(0)=*$.
\end{proof}

To consider the second case, we introduce:

\begin{definition} \label{torsion_annihilator}
We say that a  space $Z$ has torsion annihilators if it has the following properties:
\begin{enumerate}
\item $\pi_1(Z)$ is an abelian group;
\item for any given integers $n,N$, there is a self-map $g\colon Z\to Z$ such that
\begin{enumerate}
\item $g_*\otimes\Q\colon\pi_*(Z)\otimes\Q\to\pi_*(Z)\otimes\Q$ is an isomorphism, and
\item for each $i \le n$, the map $g_* \colon \pi_i(Z) \to \pi_i(Z)$ is multiplication by an integer $m_i$ with $N \mid m_i$.
\end{enumerate}
\end{enumerate}
\end{definition}

For example, $S^n \vee \R P^\infty$ is a space which has torsion annihilators but is not nilpotent.

\begin{lemma}
\label{finite2}
If $B\in\F$ has torsion annihilators and a finite complex $Z$ satisfies $H_k(Z;\Q)=0$ for $k\in q(B)$, then $[Z,B]$ is finite.
\end{lemma}

\begin{proof}
Since $Z$ is a finite complex, we have $[Z,B]\cong[Z,B_n]$ for large $n$. Then it suffices to show that $[Z, B_n]$ is finite for any $n$. To see this, we prove by induction that there is a self-map $g\colon B\to B$ such that $g$ is an isomorphism in rational homotopy groups and $(g_n)_*\colon[Z,B_n]\to[Z,B_n]$ is the constant map. When $B_0 = *$ this condition is satisfied.

Suppose that $[Z,B_{n-1}]$ is finite and there is a self-map $h\colon B \to B$ such that $h$ is an isomorphism in rational homotopy groups and $(h_i)_*\colon[Z,B_i] \to [Z,B_i]$ is the constant map for $i < n$. By the naturality of Postnikov towers, we have the following homotopy commutative diagram.
$$\xymatrix{
K(\pi_n(B),n)\ar[r]\ar[d]^{h_*}&B_n\ar[r]^{p_n} \ar[d]^{h_n}&B_{n-1}\ar[d]^{h_{n-1}}\\
K(\pi_n(B),n)\ar[r]&B_n \ar[r]^{p_n} & B_{n-1}
}$$
Then any map $f \colon Z\to B_n$ satisfies $p_n \circ h_n \circ f\simeq h_{n-1}\circ p_n\circ f\simeq*$, so $h_n \circ f$ has a lift $e\colon Z\to K(\pi_n(B),n)$, up to homotopy. By the assumption on $Z$, there is an integer $N$ such that $N\cdot H^n(Z;\pi_n(B))=0$, so $Ne=0$. Since $B$ has torsion annihilators, there is a self-map $\ell \colon B\to B$ such that $\ell$ is an isomorphism in rational homotopy groups and the map $\ell_* \colon\pi_n(B)\to\pi_n(B)$ is the multiplication by an integer $M$ with $N\mid M$. Then we see that $\ell_n \circ h_n \circ f\simeq*$ for any $f\in[Z,B_n]$. Let $F$ be the homotopy fiber of $\ell_n \circ h_n$. Then $F$ is a torsion space and $[Z,F]\to[Z,B_n]$ is surjective. Since $Z$ is a finite complex, $[Z,F]$ is a finite set, so $[Z,B_n]$ is too a finite set. This completes the proof.
\end{proof}

Now we give our answer to Problem \ref{generalized-top-dBE}.

\begin{theorem}
\label{dBE_general}
Let $s_n \colon B \to B_n$ be the $n$-th Postnikov section, and suppose that $B\in\F$ is nilpotent or has torsion annihilators. If $H_k(X;\Q)=0$ for $k\in q(B)$, then $\Ph(X,s_n)$ is trivial for any $n$.
\end{theorem}

\begin{proof}
Consider a map between the inverse systems of pointed sets $\{[X^k,B]\} \to \{[X^k,B_n]\}$ induced by the Postnikov section $s_n\colon B\to B_n$. There is a commutative diagram
$$\xymatrix{
[X, B] \ar[r]^-{\pi_B} \ar[d]_{(s_n)_*} & \lim [X^k, B] \ar[d]^{(s_n)_*}\\
\Ph(X,s_n) \ar[r]^-{\pi_{B_n}} & \lim \, (s_n)_*[X^k, B],
}$$
where the horizontal arrows are surjective by \eqref{milnor} and Proposition \ref{Ker_I}. Since $[X^k, B_n] \cong [X, B_n]$ for $k > n$, the map $\pi_{B_n} \colon [X, B_n] \to \lim [X^k, B_n]$ is injective. Then since $\Ph(X,s_n)$ is a subset of $[X,B_n]$ and the lower $\pi_{B_n}$ is the restriction of $\pi_{B_n}\colon [X, B_n] \to \lim [X^k, B_n]$, the lower $\pi_{B_n}$ is injective, so it is bijective. Then it follows that $\Ph(X, s_n)$ is trivial if and only if the right $(s_n)_* $ is surjective. Thus we shall show that the right $(s_n)_*$ is surjective.

Note that the map $(s_n)_* \colon [X^k,B] \to (s_n)_*[X^k, B]$ is surjective for any $k$ and that by Lemmas \ref{finite1} and \ref{finite2}, $[X^k,B]$ is a finite set for any $k$. It follows from Lemma \ref{surj} that the right $(s_n)_*$ is surjective as desired. This completes the proof.
\end{proof}

Finally, we deal with the case that $\varphi$ is the inclusion $i_n \colon \R P^n \hookrightarrow \R P^\infty$. Since $\R P^n$ is nilpotent for an odd $n$, Theorem \ref{dBE_general} implies the following corollary:

\begin{corollary}
\label{dBE_answer}
If $H_{2n+1}(X ; \Q) = 0$ then $\Ph(X, i_{2n+1})$ is trivial.
\end{corollary}

We finally show that Corollary \ref{dBE_answer} is optimal by giving an example of a space $X$ such that $H_n(X;\Q)\ne 0$ and there is a non-trivial relative phantom map from $X$ to $i_n\colon\R P^n\to\R P^\infty$. We will use the following simple lemma.

\begin{lemma}
\label{odd_degree}
Let $\Z/2$ act on $S^n$ by the antipodal map. 
For every odd integer $k$, there is a $\Z/2$-map $f \colon S^n \to S^n$ of degree $k$.
\end{lemma}

\begin{proof}
The case $n=1$ is trivial, and for $n > 1$, take the $(n-1)$-fold suspension of the $\Z/2$-map on $S^1$.
\end{proof}

\begin{remark}
Lemma \ref{odd_degree} implies that there is a mistake in the calculation of the homotopy set $[\R P^n,\R P^n]$ for $n$ even due to McGibbon \cite{M2}. It is calculated as follows. Consider the homotopy cofibration sequence
$$S^{n-1}\xrightarrow{p_{n-1}}\R P^{n-1}\xrightarrow{i_{n-1}}\R P^n\xrightarrow{q_n}S^n$$
where $p_{n-1}$ is the universal covering, $i_{n-1}$ is the inclusion, and $q_n$ is the pinch map to the top cell. Then for $n-k>0$ and $k>0$, there is an exact sequence of groups
\begin{multline*}
[\Sigma^{k+1}\R P^{n-k-1},\R P^n]\xrightarrow{(\Sigma^{k+1}p_{n-k-1})^*}\pi_n(\R P^n)\\
\xrightarrow{(\Sigma^kq_{n-k})^*}[\Sigma^k\R P^{n-k},\R P^n]\xrightarrow{(\Sigma^ki_{n-k-1})^*}[\Sigma^k\R P^{n-k-1},\R P^n].
\end{multline*}
Since $\pi_n(\R P^n)=\Z\{p_n\}$, $q_k\circ p_k=1+(-1)^{k+1}$ and $[\Sigma^k\R P^{n-k-1},\R P^n]=*$, we inductively get 
$$[\Sigma^k\R P^{n-k},\R P^n]\cong\begin{cases}\Z&n-k\text{ is odd}\\\Z/2&n-k\text{ is even}\end{cases}$$
where in both cases, $p_n\circ\Sigma^kq_{n-k}$ is a generator. We next consider the exact sequence of pointed sets
$$[\Sigma\R P^{n-1},\R P^n]\xrightarrow{(\Sigma p_{n-1})^*}\pi_n(\R P^n)\xrightarrow{q_n^*}[\R P^n,\R P^n]\xrightarrow{i_{n-1}^*}[\R P^{n-1},\R P^n]\xrightarrow{p_{n-1}^*}\pi_{n-1}(\R P^n)$$
where $\pi_{n-1}(\R P^n)=0$ and $[\R P^{n-1},\R P^n]=\{*,i_{n-1}\}$. Then by the above calculation, we have
$$(i_{n-1}^*)^{-1}(*)=\{*,p_n\circ q_n\}.$$
On the other hand, by considering the action of the top cell, we see that $(i_{n-1}^*)^{-1}(i_{n-1})=\{h_{2j-1}\,\vert\,j\in\Z\}$, where $h_m$ is the self-map of $\R P^n$ which lifts to the degree $m$ self-map of $S^n$ as in Lemma \ref{odd_degree}. Thus we obtain that
$$[\R P^n,\R P^n]=\{*,p_n\circ q_n,h_{2j-1}\;(j\in\Z)\}.$$
\end{remark}

For $n>2$, let $X(n)$ be the cofiber of the composite of maps
$$\bigvee_pS^{n+2p-3}\xrightarrow{\alpha_1}S^n\xrightarrow{\pi}\R P^n$$
where $p$ ranges over all odd primes and $\alpha_1\vert_{S^{n+2p-3}}$ is a generator $\alpha_1(p)$ of $\pi_{n+2p-3}(S^n)\cong\Z/p$ (see \cite{T}). By definition, we have $H^1(X(n);\Z/2)\cong\Z/2$, and let $f\colon X(n)\to\R P^\infty$ be the generator of $H^1(X(n);\Z/2)$.

\begin{proposition}
\label{dBE_non-trivial}
The map $f\colon X(n)\to\R P^\infty$ is a non-trivial relative phantom map to the inclusion $i_n\colon\R P^n\to\R P^\infty$.
\end{proposition}

\begin{proof}
Suppose that $f$ is homotopic to a map $g\colon X(n)\to\R P^n$. Then since $g$ induces an isomorphism in $\pi_1$, $g\vert_{\R P^n}$ lifts to a degree $k$ map of $S^n$ for some odd integer $k$. By definition, the composite $g\vert_{\R P^n}\circ\pi\circ\alpha_1(p)$ must be null homotopic for any odd prime $p$. Since $\alpha_1(p)$ is a co-H-map \cite{BH}, we have
$$g\vert_{\R P^n}\circ\pi\circ\alpha_1(p)\simeq\pi\circ k\circ\alpha_1(p)\simeq\pi\circ(k\alpha_1(p)).$$
Then since $\pi_*\colon\pi_*(S^n)\to\pi_*(\R P^n)$ is an isomorphism for $*\ge 2$, we get that $k\alpha_1(p)$ is null homotopic. Thus $k$ is divisible by any odd prime, which is a contradiction because $k \ne 0$ since $k$ is odd. Therefore $f$ does not lift to $\R P^n$ through the inclusion $i_n\colon\R P^n\to\R P^\infty$, up to homotopy. So if $f$ is a relative phantom map, then it is non-trivial.

Fix an odd prime $p$. By Lemma \ref{odd_degree}, for any given odd integer $k$, there is a self-map $h_k\colon\R P^n\to\R P^n$ which lifts to a degree $k$ self-map of $S^n$. Let $p_1,\ldots,p_m$ be all odd primes less than or equal to $p$. Then by the above observation, we see that the map $h_k\colon\R P^n\to\R P^n$ extends to a map $\bar{h}_k\colon X(n)\to X(n)$, and by looking at $\pi_1$, we have
$$f\simeq f\circ \bar{h}_{p_1}\circ\cdots\circ \bar{h}_{p_m}.$$
Since $h_{p_i}\circ\pi\circ\alpha_1(p)\simeq\pi\circ (p_i\alpha_1(p))$ as above, we see that the restriction of $f\circ \bar{h}_{p_1}\circ\cdots\circ \bar{h}_{p_m}$ to $X(n)^{n+2p-2}$ lifts to $\R P^n$ through $i_n$, up to homotopy. Since the prime $p$ can be arbitrary large, $f$ is a relative phantom map to the inclusion $i_n\colon\R P^n\to\R P^\infty$. Therefore we obtain that  $f$ is a non-trivial relative phantom map to $i_n\colon\R P^n\to\R P^\infty$, completing the proof.
\end{proof}

\begin{remark}
It follows from Corollary \ref{dBE_answer} that $H_{2n+1}(X ; \Q) = 0$ implies that $\Ph(X, i_{2n+1})$ is trivial. On the other hand, $H_{2n}(X(2n) ; \Q) = 0$ but $\Ph(X, i_{2n})$ is non-trivial. In fact, there is no $k$ such that $H_k(X ; \Q) = 0$ implies that $\Ph(X, i_{2n})$ is trivial. 

If such an integer $k$ exists, $k = n + 2p -2$ for some odd prime $p$ by the rational homology of $X(2n)$. Let $X'(n)$ be the subcomplex of $X(n)$ such that we delete the $n + 2p -2$-cell from $X(n)$. Then the restriction $f|_{X'(n)} \colon X'(n) \to \R P^\infty$ is a non-trivial relative phantom map to $i_n$ by the same reason.
\end{remark}

\end{document}